\documentclass{llncs}

\usepackage[utf8]{inputenc}
\usepackage[T1]{fontenc}
\usepackage[brazilian]{babel}
\usepackage{csquotes}
\usepackage{hyphenat}
\usepackage{amssymb}
\usepackage{amsmath}
\usepackage{array}

\usepackage[locale=FR]{siunitx}

\usepackage{graphicx}
\graphicspath{ {images/} }

\usepackage{listings}

\begin{document}

\title{On $p$-Dirac Equation on Compact Spin Manifolds}
%

\author{Lei Xian and Xu Yang*}
%
%
%
\institute{School of Mathematics , Yunnan Normal University\\
\email{yangxu@ynnu.edu.cn}}

\maketitle              
\renewcommand\refname{References}
\renewcommand{\abstractname}{Abstract}
\begin{abstract}
By using the Ljusternik-Schnirelman principle, we establish the existence of a nondecreasing sequence of nonnegative eigenvalues for the p-Dirac operator on compact spin manifold. Using the biorthogonal system theory on separable Banach space and  some critical point theorems, we prove the existence and  multiplicity of solutions to p-superlinear and p-sublinear nonlinear p-Dirac equations on compact spin manifold.
\end{abstract}
\section{Introduction and main results}\label{sec:1}
\setcounter{equation}{0}

Motivated by quantum physics, Esteban-S\'er\'e \cite{EsSe1, EsSe2} studied
 existence and multiplicity of  solutions of nonlinear Dirac
equations on $\mathbb{R}^3$. In this case, a large number of existence of solution has been obtained.
Dirac operators on compact spin manifolds play prominent role in the geometry and mathematical physics, such as the generalized Weierstrass representation of the surface in three manifolds \cite{Fri} and the supersymmetric nonlinear sigma model in quantum field theory \cite{CJLW1,CJLW2}.
 For the nonlinear Dirac equations on a general compact spin manifold, some results were recently obtained.
Ammann studied in \cite{Amm1,Amm2} the special from of the equation $D\psi=\lambda|\psi|^{p-2}\psi$ for $\lambda>0$
and $2<p\le\frac{2m}{m-1}$ 1 on a general compact spin manifold. By the special character of the nonlinearity$\lambda|\psi|^{p-2}\psi$,  the problem in
this case is equivalent to the “dual” variational problem and a solution of the latter is
obtained by minimizing a certain functional which is bounded. In this way Ammann
obtained a non-trivial solution for each subcritical convex case $2<p<\frac{2m}{m-1}$. He
also gave an existence criterion (which is similar to the condition given by Aubin \cite{Aub} for a resolution of the Yamabe problem) for the critical case $p=\frac{2m}{m-1}$.
Notice that this form of nonlinear Dirac equations arise as spinorial analogue of Yamabe problem \cite{Amm1,Amm2,AHM},
The problem of the existence of conformally immersed hypersurfaces in a manifold \cite{Amm1,Amm2} and spinorial analogue of the supersymmetric extension of harmonic
maps \cite{CJLW1,CJLW2}, Ammann and his collaborators’ result is recently extended to the equations of the form $D\psi=H(x)|\psi|^{p-2}\psi$
by Raulot \cite{RaSi} also by considering the dual definite functional. See also \cite{Iso2},  where the present author proved the existence
and compactness of solutions to the equation of the form $D\psi=\lambda\psi+|\psi|^{\frac{2}{m-1}}\psi$.
For general nonlinearity, Isobe\cite{Iso1} consider Dirac equations on compact spin manifold via Galerkin type approximations and linking arguments, the existence of infinitely many solutions is obtained under  Ambrosetti-Rabinowitz condition.
In addition, be different with these existing works, Nolder and Ryan introduced $p-$Dirac operators and $p-$Harmonic
sections on spin manifolds in \cite{Nol1}. They established the conformal covariance of these operators and obtained existence results for solutions to $p$-Dirac and $p$-Harmonic equation on the sphere $S^n$. In\cite{Nol3}, Nolder given the solutions to nonlinear $A$-Dirac equations of the form $DA(x,Du)=0$. $p$-Dirac equation, appears in \cite{Nol1}, is a special case of $A$-Dirac equations. $A$-Dirac equations are nonlinear generalizations of the elliptic equations of $A$-harmonic type $\text{div} A(x,\nabla u)=0$.

As for the author’s knowledge, except for these works, general existence result has not
been established for the equation $D_{p}\psi=\lambda|\psi|^{p-2}\psi$ ,
or more generally, it is natural to ask whether there exist infinitely
many solutions to the equation (\ref{eq1.1}). In this paper, we are concerned with the existence of solutions to  $p$-Dirac equations on compact spin manifold which is an analogue of $p$-Laplacian equations from a variational point of view,see \cite{BoRo,Cos,Lin,MaRo}.

Assume $m\ge 2$ is an integer.
 Given an $m$-dimensional compact oriented Riemannian manifold
  $(M,g)$ equipped with a spin structure $\rho:P_{Spin(M)}\rightarrow P_{SO(M)}$,  let $\mathbb{S}(M)=\mathbb{S}(M,g)=P_{Spin(M)} \times_{\sigma}\mathbb{S}_m$ denote
the complex spinor bundle on $M$, which is a complex vector bundle of rank $2^{[m/2]}$ endowed with the spinorial Levi-Civita connection $\nabla$ and a pointwise Hermitian scalar product $\langle\cdot,\cdot\rangle$.  Write the point of $\mathbb{S}$ as
$(x, \psi)$, where $x\in M$ and $\psi\in \mathbb{S}_x(M)$. Laplace operator on spinors $\Delta:C^\infty(M,\mathbb{S}(M))\rightarrow C^\infty(M,\mathbb{S}(M))$ is defined by
$$
\Delta\psi=-\sum_{i=1}^n(\nabla_{e_i}\nabla_{e_i}-\nabla_{\nabla_{e_i}e_i})\psi,
$$
 it is independent of the choice of local orthonormal frame  and  is a second order differential elliptic operator.
 \par
 The Dirac operator $D$ is an elliptic differential operator of order one,
 $$
 D=D_g:C^{\infty}(M,\mathbb{S}(M))\rightarrow C^{\infty}(M,\mathbb{S}(M)),
 $$
 locally given by $D\psi=\sum^m_{i=1}e_i\cdot\nabla_{e_i}\psi$ for $\psi\in
C^{\infty}(M,\mathbb{S}(M))$ and a local $g$-orthonormal frame
$\{e_i\}^m_{i=1}$ of the tangent bundle $TM$. It is an unbounded essential self-adjoint operator in $L^2(M,\mathbb{S}(M))$ with the domain $C^\infty(M,\mathbb{S}(M))$,
and its spectrum consists of an unbounded sequence of real numbers.
these mean that the closure of $D$, is a self-adjoint operator in $L^2(M,\mathbb{S}(M))$  with the domain $H^1(M,\mathbb{S}(M))$.

Similar to $p$-Laplacian operator, $p$-Dirac operator
$$D_{p}: C^{\infty}(M,\mathbb{S}(M))\to C^{\infty}(M,\mathbb{S}(M))$$
 is defined in \cite{Nol1}, where $D_{p}\psi:=D|D\psi|^{p-2}D\psi$.

If $p=2$, then $D_{2}\psi:=D^2\psi$ , The square $D^2$ of the Dirac operator as well as that of the Laplace operator $\Delta$ is second order differential operator, which is called Dirac-Laplace operator. On the square Dirac operator has been discussed a lot in the literature,c.f.\cite{Fri,LiG,WoC}. the connection between $D^2$ and $\Delta$ is given by Schr\"{o}dinger-Lichnerowicz formula
\begin{equation*}
  D^2=\Delta+\frac{R}{4}\text{Id},
\end{equation*}
where $R$ is the scalar curvature of $(M,g)$.
\par
For a fiber preserving nonlinear map $h:\mathbb{S}(M)\rightarrow \mathbb{S}(M)$,
we consider the following nonlinear $p$-Dirac equation
\begin{equation}\label{eq1.1}
    D_{p}\psi=h(x,\psi(x))\enspace\text{on} \enspace M,
\end{equation}
where the point of $\mathbb{S}(M)$ is written as $(x,\psi)$ with $x\in M$ and $\psi \in \mathbb{S}_{x}(M)$
(the fiber over the point $x$) and $\psi\in C^{\infty}(M,\mathbb{S}(M))$ is a spinor.

We assume $h$ has a $\psi$-potential, namely there exist real valued function $H:\mathbb{S}(M)\rightarrow \mathbb{R}$ such that $H_\psi=h$, then equation (\ref{eq1.1}) has a variational structure.
 It is  the Euler-Lagrange equation of the functional
\begin{eqnarray}\label{eq1.2}
\mathfrak{L}(\psi)=\frac{1}{p}\int_M\langle D_p\psi,\psi\rangle dx
-\int_M H(x,\psi(x))dx
\end{eqnarray}
on the Sobolev space $H^{1,p}(M,\mathbb{S}(M))$, see\cite{Ada}. Here  $dx$ is the Riemannian volume measure on $M$ with respect to the metric $g$, and $\langle\cdot,\cdot\rangle$ is the compatible metric on $\mathbb{S}(M)$.

From a variational point of view, the functional \eqref{eq1.2} is interesting and difficult mainly because it’s indefinite character and Lack of spectral properties with $p$-Dirac operator on compact spin manifold.
On the one hand, as far as i know, there is no results on the spectrum of $p$-Dirac operator on compact spin manifold at present. However similar to quasilinear elliptic equation $\Delta_p \psi= h(x,\psi)$ on $\Omega\subset\mathbb{R}^n$. As everyone knows, many results have been obtained on the structure of the spectrum for the $p$-Laplacian operator on a bounded domain in European space, i.e., there exists a nondecreasing sequence of nonnegative eigenvalues and the first eigenvalue is simple and isolate, the spectrum of $p$-Laplacian operator is a closed set,see \cite{LAn}. By referring to these ideas, in this note, we only
 show the existence  of a nondecreasing sequence of nonnegative eigenvalues for $p$-Dirac operator on compact spin manifold.
 On the other hand, a suitable function space to work with the functional $\mathfrak{L}$ is the separable Sobolev space $H^{1,p}(M,\mathbb{S}(M))$, which is not a Hilbert space. In order to use the critical point theorem, we introduce the theory of biorthogonal system in (\cite{LiTz},p42-43). Through the the theory of biorthogonal system, the pace $H^{1,p}(M,\mathbb{S}(M))$ can be decomposed into the form we need.

Through the article, $C$ and $C_i,i=1,2,3\cdot\cdot\cdot$ represent a positive real number, $p>1$ and  we assume that the Dirac operator $D$ is invertible. as follow, we give the existence and multiplicity results for solutions to the
equation \eqref{eq1.1} for a general class of nonlinearities $h=H_{\psi}$.
In order to deal with such a general class of nonlinearities, we shall work directly with the strongly indefinite functional \eqref{eq1.2}. Henceforth,
we assume $H\in C^{0}(\mathbb{S}(M))$ and it is continuously differentiable in the
fiber direction, $i.e.$, $H_{\psi}\in C^{0}(\mathbb{S}(M))$.

On the one hand, we consider the p-superlinear case, we shall work on the following class of nonlinearities $H$:
\begin{enumerate}
    \item[$(H1)$]  For any $(x,\psi)\in \mathbb{S}(M)$, there exist $q\in (p,p^*)$ and $C_{1}>0$ such that
    \begin{equation}
        |H_{\psi}(x,\psi)|\le C_{1}\left(1+|\psi|^{q-1}\right)
        \nonumber
    \end{equation}
    where $p^*=\frac{mp}{m-p}$.
    \item[$(H2)$] There exist $\mu>p$ and $R_{1}>0$ such that
    \begin{equation}
        0<\mu H(x,\psi)\le \left\langle H_{\psi}(x,\psi),\psi\right\rangle
        \nonumber
    \end{equation}
    for any $(x,\psi)\in\mathbb{S}(M)$ with $|\psi|\ge R_{1}$.
    \item[$(H3)$] $H(x,\psi)=o(|\psi|^{p})$ as $|\psi|\to0$ uniformly for $x\in M$.
    \item[$(H4)$] $H(x,\psi)=H(x,-\psi)$ for any $(x,\psi)\in \mathbb{S}(M)$.
\end{enumerate}

our first existence result for the p-superlinear case is the following:

\begin{theorem}\label{th1.1}
    Assume $H$  satisfies $(H1), (H2), (H3)$.
    Then there exists a weak solution $\psi\in H^{1,p}(M,\mathbb{S}(M))$ to the equation \eqref{eq1.1}.
\end{theorem}

Notice that $H_{\psi}(x,0)=0$ by (H4) and $\mathfrak{L}(0)=0$. Thus \eqref{eq1.1} has the trivial
solution $\psi\equiv 0$ under the assumption of Theorem \ref{th1.1}. Theorem \ref{th1.1} guaranties
the existence of a non-trivial solution.

For odd nonlinearities, we have the following multiplicity results:

\begin{theorem}\label{th1.2}
Assume $H$ satisfies $(H1), (H2), (H3), (H4)$.
Then there exists a sequence of weak solutions $\left\{\psi_{k}\right\}_{k=1}^{\infty}\subset H^{1,p}(M,\mathbb{S}(M))$ to \eqref{eq1.1}
 with $\mathfrak{L}(\psi_{k})\to\infty$ as $k\to\infty$.
\end{theorem}

On the other hand, we consider the p-sublinear case. the following conditions on $H(x,\psi)$ are used:
\begin{enumerate}
    \item[$(Hi)$] For any $(x,\psi)\in \mathbb{S}(M)$, there exist $1<q<p$ and $C_{2}>0$ such that
    \begin{equation}
        |H_{\psi}(x,\psi)|\le C_{2}\left(1+|\psi|^{q-1}\right).
        \nonumber
    \end{equation}

    \item[$(Hii)$] There exist $1<\nu<p$ and $C_3>0$ such that
    \begin{equation}
        C_3|\psi|^\nu \leq H(x,\psi)
        \nonumber
    \end{equation}
    for any $(x,\psi)\in\mathbb{S}(M)$.
\end{enumerate}

\begin{theorem}\label{th1.3}
Assume $H$ satisfies $(Hi)$.
Then $c:=\inf\{\mathfrak{L}(\psi):\psi\in H^{1,p}(M,\mathbb{S}(M))\}$ is a critical value of $\mathfrak{L}$.
\end{theorem}

For odd nonlinearities, we have the following multiplicity results:
\begin{theorem}\label{th1.4}
Assume $H$ satisfies $(Hi), (Hii), (H4)$.
Then there exists a sequence of weak solutions $\left\{\psi_{k}\right\}_{k=1}^{\infty}\subset H^{1,p}(M,\mathbb{S}(M))$ to \eqref{eq1.1}
 with $\mathfrak{L}(\psi_{k})<0$ for all $k$ and $\mathfrak{L}(\psi_{k})\to 0$ as $k\to\infty$.
\end{theorem}

\section{Preliminaries}
\subsection{Spin structures and Dirac operators}

We collect here basic definitions and facts about spin structures on manifolds and
Dirac operators. Since our purpose is only to introduce notation which are used
throughout this paper, we do not enter this subject in detail. For more detailed
exposition, please consult \cite{LaM,Fri}.

Let $M$ be an $m-$dimensional oriented Riemannian manifold. We henceforth
assume that $m\ge2$. Since $M$ is orientable, the tangent bundle $TM$ admit an $SO(m)-$ structure:
it can be defined by an open covering $\left\{U_{\alpha}\right\}$ of M and transition
maps $g_{\alpha \beta}: U_{\alpha\beta}:=U_{\alpha}\cap U_{\beta}\to SO(m)$ satisfying $g_{\alpha\alpha}=1$ and the cocycle
condition: $g_{\alpha\beta}\cdot g_{\beta\gamma}\cdot g_{\gamma\alpha}=1$ in $U_{\alpha\beta\gamma}:=U_{\alpha}\cap U_{\beta}\cap U_{\gamma}$,
where $1\in SO(m)$ is the identity. Recall that $SO(m)$
is not simply connected (indeed, $\pi_{1}(SO(2))=\mathbb{Z}$ and $\pi _{1}(SO(m))=\mathbb{Z}_{2}$ for $m\ge 3$. )
Thus there exists the universal covering $ \rho: Spin(m)\to SO(m)$ for the case $m\ge3$
and the double covering $\rho : Spin(2)\cong S^{1}\ni z\mapsto z^{2}\in S^{1} \cong SO(2)$ for the case $m=2$. The manifold $M$ is said to
possess a spin structure if there exist smooth maps $\tilde{g}_{\alpha\beta}:U_{\alpha\beta}\to Spin(m)$ satisfying $\tilde{g}_{\alpha\alpha}=1$
and the cocycle condition $\tilde{g}_{\alpha\beta}\cdot \tilde{g}_{\beta\gamma}\cdot \tilde{g}_{\gamma\alpha}=1$
in $U_{\alpha\beta\gamma}$ (here $1$ also denotes the identity element of $Spin(m)$). We will not distinguish it from the
identity of $SO(m)$ and $ \rho(\tilde{g}_{\alpha\beta})=g_{\alpha\beta}$ for all $\alpha$, $\beta$. A pair of manifold and its spin
structure is called a spin manifold. There is a topological obstruction for the existence of a spin structure, namely, the vanishing of the second Stiefel-Whitney class
$w_{2}(M)\in H^{2}(M;\mathbb{Z}_{2})$. Moreover, there may be many different spin structures on
the same manifold.

For a spin manifold $M$, $\left\{\tilde{g}_{\alpha\beta}\right\}$ define a principal $Spin(m)-$ bundle which we
denote by $P_{Spin}(M)$.  It is a double covering of the oriented frame bundle $P_{SO}(M)$ whose restriction to each fiber is $\rho: Spin(m)\to SO(m)$.
We can regard $P_{SO}(M)$ as a bundle associated to $ P_{Spin}(M)$ via $\rho: Spin(m)\to SO(m)$.

In order to introduce the spinor bundle, we first assume that $m$ is even. Recall
that the Clifford algebra $\text{Cl}_{m}$ is the associative $\mathbb{R}-$algebra with unit, generated by $\mathbb{R}^{m}$
subject to the relations $uv+vu=-2(u,v)$ for $u,v\in\mathbb{R}^{m}$ ( $(u,v)$ is the Euclidean
inner product of $u$ and $v$ in $\mathbb{R}^{m}$ ). $Spin(m)$ is the group generated multiplicatively
by even number of unit vectors in $\mathbb{R}^{m}$. There exists a complex $Cl_{m}-$module $\mathbb{S}_{m}$
such that $\mathbb{C}l_{m}:=\text{Cl}_{m}\otimes \mathbb{C}\cong \text{End}_{\mathbb{C}(\mathbb{S}_{m})}$ as $\mathbb{C}-$algebras. $\mathbb{S}_{m}$ is the unique (up to
isomorphism) irreducible complex $\text{Cl}_{m}-$module, usually called the spinor module.
The isomorphism $\mathbb{C}l_{m}\cong \text{End}_{\mathbb{C}}(\mathbb{S}_{m})$ induces the representation (unique up to isomorphism) $\sigma: Spin(m)\to\text{End}(\mathbb{S}_{m})$,
the spinor representation. On the other hand,the orientation on $\mathbb{R}_{m}$ induces a $\mathbb{Z}_{2}-$grading on $\mathbb{S}_{m}$;
$\mathbb{S}_{m}=\mathbb{S}_{m}^{+}\oplus \mathbb{S}_{m}^{-}$. Since $Spin(m)\subset\text{Cl}_{m}^{\text{even}}$
(the subalgebra of $\text{Cl}_{m}$ generated by themultiples of even number of vectors in $\mathbb{R}^{m}$),
each of the spinor spaces $\mathbb{S}_{m}^{\pm}$ is a representation space for $Spin(m)$.
They are in fact irreducible, non-isomorphic complex
$Spin(m)-$modules. They are called positive or negative complex spin representations
and we denote them as $\sigma^{\pm}:Spin(m)\to\text{End}(\mathbb{S}_{m}^{\pm})$. sociated to these, we obtain
Hermitian vector bundles:
\begin{equation}
    \begin{aligned}
        \mathbb{S}(M)&:= P_{SPin}(M)\times _{\sigma}\mathbb{S}_{m},
        \\
        \mathbb{S}^{\pm}(M)&:=P_{Spin}(M)\times_{\sigma^{\pm}}\mathbb{S}_{m}^{\pm}.
    \end{aligned}
    \nonumber
\end{equation}
These are called complex (positive/negative) spinor bundles.

For the odd $m$ case, we first observe that there is an isomorphism $\text{Cl}_{m}\cong\text{C}l_{m+1}^{\text{even}}$ defined by $x^{0}+x^{1}\mapsto x^{0}+e_{0}\cdot x^{1}$,
where $x^{0}\in \text{C}l_{m}^{even}$, $x^{1}\in \text{C}l_{m}^{odd}$ and $\left\{e_{0},e_{1},\cdots, e_{m}\right\}$
and $\left\{e_{1},\cdots, e_{m}\right\}$ are orthonormal bases of $\mathbb{R}^{m+1}$ and $\mathbb{R}^{m}$, respectively.
(Thus we shall regard $\mathbb{R}^{m}$ as a subspace of $\mathbb{R}^{m+1}$) We then have the isomorphism
\begin{equation}
    \mathbb{C}l_{m}\cong \mathbb{C}l_{m+1}^{\text{even}}\cong\text{End}^{\text{even}}(\mathbb{S}_{m+1})\cong\text{End}(\mathbb{S}_{m+1}^{+})\oplus \text{End}(\mathbb{S}_{m+1}^{-}).
    \nonumber
\end{equation}
Thus both of $\mathbb{S}_{m+1}^{\pm}$ are representation spaces of $\mathbb{C}l_{m}$. In fact, they are both irreducible and also become representation spaces of $Spin(m)$.
It is known that they are irreducible but isomorphic as $Spin(m)-$modules. We denote $\mathbb{S}_{m}\cong\mathbb{S}_{m+1}^{+}$ and call it as the complex spinor representation.
Denoting by $\sigma: Spin(m)\to \mathbb{S}_{m}$ the representation so obtained, as in the even case, we form the spinor bundle $\mathbb{S}(M)$ as
\begin{equation}
    \mathbb{S}(M):=P_{Spin}(M)\times _{\sigma}\mathbb{S}_{m}.
    \nonumber
\end{equation}
We denote by $\left\langle\cdot,\cdot\right\rangle$ the real scalar product on $\mathbb{S}(M)$ ($i.e.$, the real part of the Hermitian product on $\mathbb{S}(M)$).
The sections of the spinor bundle $\mathbb{S}(M)$ are simply called spinors on $M$.

To define the (Atiyah-Singer-)Dirac operator, we recall that on $\mathbb{S}(M)$ there is a
natural connection $\bigtriangledown :C^{\infty}(M,\mathbb{S}(M))\to C^{\infty}(M,T^{\ast}M\otimes\mathbb{S}(M))$ induced from the
Levi-Civita connection on $TM$. We also call it as the Levi-Civita connection on $\mathbb{S}(M)$.
The Dirac operator $D$ acts on spinors on $M$ , $D:C^{\infty}(M,\mathbb{S}(M))\to C^{\infty}(M,\mathbb{S}(M))$, and is defined by

\begin{equation}
    \begin{aligned}
        D&=c\circ \bigtriangledown:C^{\infty}(M,\mathbb{S}(M))\overset{\bigtriangledown}{\rightarrow} C^{\infty}(M,T^{\ast}M\otimes\mathbb{S}(M))
        \\
        &\cong C^{\infty}(M,TM\otimes\mathbb{S}(M))\overset{c}{\rightarrow} C^{\infty}(M,\mathbb{S}(M)),
    \end{aligned}
    \nonumber
\end{equation}
where $c$ denotes the Clifford multiplication $c:TM\otimes\mathbb{S}(M)\ni X\otimes\psi\mapsto X\cdot\psi\in\mathbb{S}(M)$
and we have used the identification $T^{\ast}M\cong TM$ by the metric on $M$.

For the even dimensional case, we have $D^{\pm}: C^{\infty}(M,\mathbb{S}^{\pm}(M))\to C^{\infty}(M,\\\mathbb{S}^{\mp}(M))$,
where $D^{\pm}$ is the restriction of $D$ to $C^{\infty}(M,\mathbb{S}^{\pm}(M))$ and
$D=\\\begin{pmatrix}
    0&D^{-} \\
    D^{+}&0
  \end{pmatrix}:C^{\infty }(M,\mathbb{S}^{+}(M))\oplus C^{\infty}(M,\mathbb{S}^{-}(M))\to C^{\infty}(M,\mathbb{S}^{+}(M))\oplus C^{\infty}(M,\\\mathbb{S}^{-}(M))$.

\subsection{Variational setting}
To treat the problem \eqref{eq1.1} from a variational point of view, it is necessary to give a suitable functional analytic framework. A suitable function space to work with the functional
$\mathfrak{L}$ is the Sobolev space $H ^{1,p}\left ( \mathrm {M} ,\mathbb{S} \left ( \mathrm {M} \right ) \right ) $ of $H ^{1,p}$-spinors which we now define.
According to \cite{Fri,LaM}, we next give the definition of Laplace operator on spinors:
\begin{definition}\label{def:2.1}
If $\psi\in C^\infty(M,\mathbb{S}(M))$, then $\Delta:C^\infty(M,\mathbb{S}(M))\rightarrow C^\infty(M,\mathbb{S}(M))$ is defined by
$$
\Delta\psi=-\sum_{i=1}^n(\nabla_{e_i}\nabla_{e_i}-\nabla_{\nabla_{e_i}e_i})\psi.
$$
\end{definition}
Using Stokes'theorem we have
$$
\int_M\langle \Delta\psi,\psi\rangle=\int_M\langle \nabla\psi,\nabla\psi\rangle=\int_M\langle\psi,\Delta\psi\rangle
$$
for any $\psi\in C^\infty(M,\mathbb{S}(M))$. Here, $\langle \nabla\psi,\nabla\psi\rangle$ is the scalar product on $1$-forms,i.e.
$\langle \nabla\psi,\nabla\psi\rangle=\sum\limits_{i=1}^n\langle \nabla_{e_i}\psi,\nabla_{e_i}\psi\rangle$, where $\{e_i\}_{i=1}^\infty$ is a local
orthonormal frame on the spin manifold.

The first Sobolev norm \cite{Fri} of a smooth spinor field $\psi\in C^\infty(M,\mathbb{S}(M))$ is given by
$$
\parallel \psi\parallel_{H^{1.p}}^p=\parallel\psi\parallel_p^p+\parallel\nabla\psi\parallel_p^p
$$
where
$\parallel\cdot\parallel_p^p:=\int_M\mid\psi\mid^pdx $.

The corresponding Sobolev space $H^{1.p}(M,\mathbb{S}(M))$ is defined as being the completion of $\psi\in C^\infty(M,\mathbb{S}(M))$ with respect to this norm. The elements of  the space $H^{1.p}(M,\mathbb{S}(M))$ are called $H ^{1,p}$-spinors.

However, since our problem involves the Dirac operator, it would be more convenient to solve the nonlinear equation studied in this paper if another Sobolev norm is given  by
$$
\parallel \psi\parallel_{1,p}^p=\parallel D\psi\parallel_p^p.
$$
Thus we will obtain the next Lemma:
\begin{lemma}\label{lem:2.2}
The norm $\parallel \psi\parallel_{H^{1,p}}$ and the norm $\parallel \psi\parallel_{1,p}$ are equivalent norms on $H^{1,p}(M,\mathbb{S}(M))$.
\end{lemma}
proof: Firstly, from the estimate below for the norm of $D\psi$:
\begin{eqnarray*}
\parallel D\psi\parallel_p^p &=& \sum\limits_{i,j=1}^n\int_M\mid\langle e_i\cdot\nabla_{e_i}\psi,e_j\cdot\nabla_{e_j}\psi\rangle\mid^{\frac{p}{2}} dx\leq
\sum\limits_{i,j=1}^n\int_M(\mid\nabla_{e_i}\psi\mid\mid\nabla_{e_j}\psi\mid)^{\frac{p}{2}} dx \\
   &\leq& (\frac{1}{2})^{\frac{p}{2}}\sum\limits_{i,j=1}^n\int_M(\mid\nabla_{e_i}\psi\mid+\mid\nabla_{e_j}\psi\mid)^pdx\\
    &\leq& 2^{\frac{p}{2}-1}\sum\limits_{i,j=1}^n\int_M(\mid\nabla_{e_i}\psi\mid^p+\mid\nabla_{e_j}\psi\mid^p)dx\\
   &=&2^{\frac{p}{2}}n\parallel \nabla\psi\parallel_p^p\leq 2^{\frac{p}{2}}n\parallel\psi\parallel_{H^{1,p}}^p,
\end{eqnarray*}
which implies the existence of a positive constant $C$ such that
$$
\parallel \psi\parallel_{1,p}\leq C\parallel \psi\parallel_{H^{1.p}}.
$$

On the other hand, there also exists another positive constant $C$ such that
$$
\parallel \psi\parallel_{H^{1.p}}\leq C\parallel \psi\parallel_{1,p},
$$
see (\cite{RaSi},p1592-1593).
Thus $\parallel \psi\parallel_{H^1}$ and $\parallel \psi\parallel_{1,p}$ are equivalent norms.

In other words,
the Sobolev space $H^{1,p}(M,\mathbb{S}(M))$ can be defined as the completion of the $C^\infty(M,\mathbb{S}(M))$ with respect to the norm $\parallel \psi\parallel_{1,p}$. In the following, we will use the norm $\parallel \psi\parallel_{1,p}$ respect to the $H^{1,p}(M,\mathbb{S}(M))$.

By the Sobolev embedding theorem , we have the continuous embedding
$$H^{1,p}(M,\mathbb{S}(M))\subset L^{q}(M,\mathbb{S}(M))$$
for $1\le q\le\frac{mp}{m-p}$.
Moreover, it is compact if $1\le q<\frac{mp}{m-p}$. In view of this, in order to obtain a well-defined Lagrangian $\mathfrak{L}$ on $H^{1,p}(M,\mathbb{S}(M))$,
we need to assume the following growth codition on $H : H$ is continuous and
\begin{equation}\label{eq2.3}
    |H(x,\psi)|\le C\left(1+|\psi|^{q}\right)
\end{equation}
for any $(x,\psi)\in\mathbb{S}(M)$, where $1\le q \le\frac{mp}{m-p}$ and $C$ is a positive constant.

Under \eqref{eq2.3}, by the Sobolev embedding theorem, $\mathfrak{L}$ is well-defined and continuous on $H^{1,p}(M,\mathbb{S}(M))$.
However, in order to apply critical point theory for $\mathfrak{L}$, we need a further regularity of $\mathfrak{L}$. Under the condition as follow, by  Sobolev embedding theorem, it is easily checked that $\mathfrak{L}$ is a $C^1$ functional on $H^{1,p}(M,\mathbb{S}(M))$ and we obtain the following proposition:

\begin{proposition}
Assume that  $H\in C^0(M,\mathbb{S}(M))$ is $C^1$ in the fiber direction, i.e., $H_\psi(x,\psi)\in C^0(M,\mathbb{S}(M))$, and there exist $q\in (1,p^*)$ and $C>0$ such that
    \begin{equation}\label{eq2.4}
        |H_{\psi}(x,\psi)|\le C\left(1+|\psi|^{q-1}\right)
    \end{equation}
for any $(x,\psi)\in \mathbb{S}(M)$.
    Then functional  $\mathcal{H}:H^{1,p}(M,\mathbb{S}(M))\to \mathbb{R}$
 defined by
 \begin{equation} \label{eq2.5}
\mathcal{H}(\psi)=\int_MH(x,\psi)dx,
\end{equation}
 is of class $C^1$, and  at each $\psi\in H^{1,p}(M,\mathbb{S}(M))$, (Fr$\acute{e}$chet) derivations $\mathcal{H}^\prime(\psi)$
 is given by
\begin{equation} \label{eq2.6}
\mathcal{H}^\prime(\psi)\xi=\int_M\langle H_\psi(x,\psi),\xi\rangle dx\quad\forall \xi\in H^{1,p}(M,\mathbb{S}(M)).
\end{equation}
\end{proposition}
Refer to Appendix proposition \ref{prop:A.1.} for the proof process of this proposition.

In the view of the calculus of variations, the weak solutions to the problem (\ref{eq1.1}) are obtained as critical points of the following Euler-Lagrange functional
\begin{eqnarray*}
\mathfrak{L}(\psi)=\frac{1}{p}\int_M\mid D\psi\mid^p dx.
-\int_M H(x,\psi(x))dx
\end{eqnarray*}

In fact, for $\xi \in \mathrm {C}^{\infty } (M ,\mathbb{S}(M))$, we have
\begin{equation}\label{eq2.7}
\begin{aligned}
    \left \langle \mathrm{d}\mathfrak{L}\left ( \psi  \right ),\xi \right \rangle
    &= \int_M \left \langle \left | D \psi \right | ^{p-2}D \psi,D\xi \right \rangle dx- \int_{M}\left \langle H_{\psi } \left(x,\psi  \right ),\xi\right \rangle dx\\
    &= \int_{M} \left \langle D\left | D \psi\right | ^{p-2}D\psi,\xi\right \rangle dx-\int_{\mathrm {M}}\left \langle H_{\psi } \left ( x,\psi  \right ), \xi \right \rangle dx.
    \end{aligned}
\end{equation}
It follows from \eqref{eq2.7} that $ \mathrm{d}\mathfrak{L}\left ( \psi  \right ) = 0$ is equivalent to the problem \eqref{eq1.1} with $h = H_{\psi } $ as asserted.

\subsection{The existence of L-S eigenvalues sequence}
In the part of this paper we consider the special case of the problem \eqref{eq1.1}: $h(x,\psi)=\lambda\left | \psi  \right | ^{p-2} \psi$, i.e.,
\begin{equation}\label{eq2.8}
    D_{p}\psi =\lambda\left | \psi  \right | ^{p-2} \psi \quad \text{on} \quad M.
\end{equation}
Especially, when $p=2$, the problem \eqref{eq2.8} becomes Dirac-Laplacian equation $D^2 \psi =\lambda\psi$. The eigenvalues of $D^2$  have been widely studied.
However, for general $p$,  there is nothing research results about the eigenvalues problem of \eqref{eq2.8} on compact spin manifold. In this part, we mainly point out that this equation \eqref{eq2.8} haave some similar properties of eigenvalues  with $p$-Laplacian.

\begin{definition}
     We define that $\psi\in H^{1,p}(M,\mathbb{S}(M))$, $|\psi|\not\equiv 0$, is an eigenfunction for \eqref{eq2.8}, if
     \begin{equation}\label{eq2.9}
         \int_{M}\langle|D\psi|^{p-2}D\psi, D\varphi\rangle dx=\lambda\int_{M}\langle|\psi|^{p-2}\psi,\varphi\rangle dx
     \end{equation}
     Whenever $\varphi\in C^{\infty}(M,\mathbb{S}(M))$. The corresponding real number $\lambda$ is called an eigenvalue.
\end{definition}

 Now we will use the Ljusternik-Schnivelman principle to establish the existence of a sequence of eigenvalue for
 eigenvalue peoblems for $p$-Dirac  operator. We give the famous Ljusternik-Schnivelman principle blow \cite{LAn,ZeEb}.

 Suppose $X$ is a reflexive Banach space and $F,G$ are two continuous functionals on $X$. Denote $S_{G}$ is the
 level $S_{G}=\left\{\psi\in X : G(\psi)=1\right\}$.
\begin{enumerate}
    \item[$(F1)$] $F,G:X\to R$ are even functionals and $F,G\in C^{1}(X,R)$ with $F(0)=G(0)=0$.
    \item[$(F2)$] ${F}'$ is strongly continuous and $\left\langle{F}'(\psi),\psi\right\rangle=0$, $\psi\in \overline{coS_{G}}$ implies $F(\psi)=0$
    where $\overline{coS_{G}}$ is the closed convex hull of $coS_{G}$.
    \item[$(F3)$] ${G}'$ is continuous, bounded and satisfies following condition $(S_0)$, $i.e.$, as $n\to \infty$ ,$\psi_{n}\rightharpoonup\psi$,
    ${G}'(\psi_{n})\rightharpoonup v$, $\left\langle{G}'(\psi_{n}),\psi_{n}\right\rangle\rightarrow\left\langle v,\psi\right\rangle$
    implies $\psi_{n}\to \psi$.
    \item[$(F4)$] The level set $S_{G}$ is bounded and $\psi\ne 0 $ implies $\left\langle{G}'(\psi),\psi\right\rangle>0$,

    $\lim_{t\to +\infty}G(t\psi)=+\infty$, $\inf_{\psi\in S_{G}}\left\langle{G}'(\psi),\psi\right\rangle>0$.
\end{enumerate}
For any positive integer $n$, denote by $A_{n}$ the class of all compact, symmetric subsets $K$ of $S_{G}$ such that $F(\psi)>0$
on $K$ and $\gamma(K)\ge n$, where $\gamma(K)$ denotes the genus of $K$,
$i.e.\gamma(K):=\inf \left\{k\in N: \exists g:K\to R^{k}\setminus 0 \right\}$ such that $g$ is continuous and odd. we define:
\begin{align*}
    a_{n}=\left\{\begin{array}{ll}
        \sup_{\Omega\in A_{n}} \inf_{\psi\in \Omega} F(\psi), &A_{n}\ne \emptyset,\\
        0,&A_{n}=\emptyset,
      \end{array}\right.
\end{align*}
and let
\begin{align*}
    \chi =\left\{\begin{array}{ll}
        \sup \{n\in N: a_{n}>0\}, & \text{if}\enspace a_{1}>0, \\
        0,&\text{if}\enspace a_{1}=0.
      \end{array}\right.
\end{align*}

\begin{theorem}\label{th2.5}
    Under assumptions $(F1)-(F4)$, for the eigenvalue problem
    \begin{equation}\label{eq2.10}
        {F}'(\psi)=\mu {G}'(\psi),\enspace \psi\in S_{G},\enspace\mu \in R,
    \end{equation}
    we have
    \begin{enumerate}
        \item[$(1)$] If $a_{n}>0$, then \eqref{eq2.10} possesses a pair $\pm \psi_{n}$ of eigenvectors and an eigenvalues $\mu_{n}\ne 0$;
        furthermore $F(\psi_{n})=a_{n}$.
        \item[$(2)$] If $\chi =\infty $, \eqref{eq2.10} has infinitely many pairs $\pm\psi$ of eigenvectors corresponding to nonzero eigenvalues.
        \item[$(3)$] $\infty>a_{1}\ge a_{2}\ge \cdots \ge 0$ and $a_{n}\to 0 $ as $n\to \infty$.
        \item[$(4)$] If $\chi=\infty$ and $F(\psi)=0$, $\psi\in \overline{cos_{G}}$ implies $\left\langle{F}'(\psi),\psi\right\rangle=0$,
        then there exists an infinite sequence $\mu_{n}$ of distinct eigenvalues of \eqref{eq2.10} such that $\mu_{n}\to 0$ as $n\to \infty$.
        \item[$(5)$] Assume that $F(\psi )=0$, $\psi\in \overline{coS_{G}}$ implies $\psi=0$. Then $\chi=\infty$ and there exists a sequence of
        eigenpairs $\left\{(\psi_{n}, \mu_{n})\right\}$ of \eqref{eq2.10} such that $\psi_{n}\rightharpoonup 0$,
        $\mu_{n}\to0$ as $n\to\infty$ and $\mu_{n}\ne 0 $ for all $n$.
    \end{enumerate}
\end{theorem}

Suppose that
\begin{equation}\label{eq2.11}
    F(\psi)=\frac{1}{p}\int_{M}|\psi|^{p}dx,
\end{equation}
\begin{equation}\label{eq2.12}
    G(\psi)=\frac{1}{p}\int_{M}|D\psi|^{p}dx,
\end{equation}
denote $A=F'$, $B=G'$, then we have
\begin{equation}\label{eq2.13}
    \left\langle A\psi, \varphi\right\rangle=\int_{M}|\psi|^{p-2}\langle\psi,\varphi\rangle dx,
\end{equation}
\begin{equation}\label{eq2.14}
    \left\langle B\psi, \varphi\right\rangle=\int_{M}|D\psi|^{p-2}\langle D\psi,D\varphi\rangle dx.
\end{equation}
It is easy to see from \eqref{eq2.13} ,\eqref{eq2.14} that $F,G$ satisfy hypotheses $(F1)$ and $(F4)$. Now we give the proofs of hypotheses $(F2)$ and $(F3)$ as follow.

In order to proving $(F2)$, We extend the lemma in (\cite{RaRe},Theorem12,p.83) to the  following lemma:

\begin{lemma}\label{lemma2.6}
 Let $(M,g)$ be an $m$-dimensional compact spin manifold and $\Phi:\mathbb{R}^+\to \mathbb{R}^+$  be a Young function which satisfies
    a  $\Delta_2$-condition,i.e., there is $c>0$ such that $\Phi(2t)\leq c\Phi(t)$ for all $t\geq 0$. If ${u_n}$ is a sequence of integrable spinor in $\mathbb{S}(M)$  and $u\in \mathbb{S}(M)$such that
    $$
    \lim_{n\to \infty}u_n(x)=u(x), a.e. \ x\in M.\quad \int_M\Phi(|u|)dv_g=\lim_{n\to \infty}\int_M\Phi(|u_n|)dv_g,
    $$
    then
    $$
    \lim_{n\to \infty}\int_M\Phi(|u_n-u|)=0.
    $$
\end{lemma}
\begin{proof}
In order to distinguish volume measures under different measures, we denote here $dv_g$ that is the Remannian volume measure on $(M,g)$.

Let $x\in M$  and $\varepsilon>0$. Let $U$(resp. $V$) be a neighborhood of $x\in M$(resp. $0\in \mathbb{R}^n$) such that the exponential map
\begin{equation}\label{eq2.15}
  \exp_x:V\subset(\mathbb{R}^n,g_{eucl})\to U\subset (M,g)
\end{equation}
is a diffeomorphism for every $x\in M$, where $g_{eucl}$ stands for the Euclidean scalar product. Then we can identify the spinor bundle over $(U,g)$ with the one over $(V,g_{eucl})$ that is there exists a map:
\begin{equation}\label{eq2.16}
  \tau:\mathbb{S}(U,g)\to \mathbb{S}(V,g_{eucl}),\psi(y)\to (\tau \psi)(\exp_x^{-1}(y)),
\end{equation}
which is a fiberwise isometry.

Since $M$ is compact, we choose a finite sequence $\{x_1,x_2,\cdot\cdot\cdot,x_{N_0}\}\subset M$ and a finite cover $\{U_i\}_{i=1}^{N_0}$ of $M$
 such that there exist open sets $\{V_i\}_{i=1}^{N_0}$ of $0\in \mathbb{R}^n$, where $x_i\in U_i$. Moreover, we set $\tau_i$ such that \eqref{eq2.15} and \eqref{eq2.16} are fulfilled.

 Let $\eta_i,(i=1,2,\cdot\cdot\cdot,N_0)$ be a smooth partition of unity subordinate to the covering $\{U_i\}_{i=1}^{N_0}$, i.e.,
 $$
 \overline{\{x\in M|\eta_i(x)\neq 0\}}\subset U_i,\quad 0\leq \eta_i\leq 1,\quad \sum_{i=1}^{N_0}\eta_i=1.
 $$
 Hence, for any $u\in \mathbb{S}(U,g)$,  we have that
\begin{eqnarray}\label{eq2.17}
  \int_M\Phi(|u|)dv_g &=& \sum_{i=1}^{N_0}\int_{U_i}\eta_i \Phi(|u|)dv_g \\ \nonumber
  &=& \sum_{i=1}^{N_0}\int_{V_i}\eta_i \Phi(|(\tau u)(\exp_{x}^{-1}(y))|) det(g_{kl}(x))^\frac{1}{2}dv_{g_{eucl}} .
\end{eqnarray}
By \eqref{eq2.17}, We just need to prove the lemma A on domain $V\subset \mathbb{R}^n$.  Since this assertion is equivalent to the result in  (\cite{RaRe},Theorem12,p.83). Therefore we complete the proof.
\end{proof}

\begin{lemma}\label{lemma2.7}
    Let $F$ be an defined as in \eqref{eq2.11}, then $F'$ satisfies $(F2)$.
\end{lemma}
\begin{proof}
    We only need to prove $A$ is strongly continuous, $i.e.$ if $\psi_{n}\rightharpoonup\psi$ in $X:=H^{1,p}(M,\mathbb{S}(M))$,
    then $A\psi_{n}\to A\psi$ in $X^{*}$.
    \begin{align}
        |\left\langle A\psi_{n}-A\psi,\varphi\right\rangle|=&\left|\int_{M}\langle\psi_{n}|^{p-2}\psi_{n}-|\psi|^{p-2}\psi,\varphi\rangle dx\right|
        \nonumber\\
        \le& \left\||\psi_{n}|^{p-2}\psi_{n}-|\psi|^{p-2}\psi\right\|_{\frac{p}{p-1}}\left\|\varphi\right\|_{p}
        \nonumber\\
        \le& C\left\||\psi_{n}|^{p-2}\psi_{n}-|\psi|^{p-2}\psi\right\|_{\frac{p}{p-1}}\left\|\varphi\right\|_{1,p}
        \nonumber
    \end{align}
    If we can prove that $|\psi_{n}|^{p-2}\psi_{n}\to |\psi|^{p-2}\psi$ in $L^{\frac{p}{p-1}}(M)$, then we complete the proof. To
     see this, let $w_n=|\psi_{n}|^{p-2}\psi_{n}$ and $w=\psi|^{p-2}\psi$.
    Since $\psi_{n}\rightharpoonup\psi$ in $H^{1,p}(M,\mathbb{S}(M))$,
 $\psi_{n}\to \psi$ in $L^{p}(M)$ and $w_n\rightarrow w$,a.e. in $M$.

 From
 $\int_M|\psi_n|^pdx-\int_M|\psi_n-\psi|^pdx\rightarrow  \int_M|\psi|^pdx$, we have
 $$\int_M|w_n|^{\frac{p}{p-1}}dx=\int_M|\psi_n|^pdx\rightarrow  \int_M|w|^{\frac{p}{p-1}}dx=\int_M|\psi|^pdx.$$
  We conclude from lemma\ref{lemma2.6}  that $|\psi_{n}|^{p-2}\psi_{n}\to |\psi|^{p-2}\psi$ in $L^{\frac{p}{p-1}}(M)$.
\end{proof}

In order to verify $(F3)$ we need the following lemma:
\begin{lemma}\label{lemma2.8}
    Let $B$ be defined as in \eqref{eq2.14}, then
    \begin{equation}
        \left\langle B\psi-B\varphi,\psi-\varphi\right\rangle\ge (\left\|\psi\right\|^{p-1}_{1,p}-\left\|\varphi\right\|^{p-1}_{1,p})(\left\|\psi\right\|_{1,p}-\left\|\varphi\right\|_{1,p}).
        \nonumber
    \end{equation}
Furthermore,$\langle B\psi-B\varphi,\psi-\varphi\rangle=0$ if and only if $u=v$ a.e. in $M$.
\end{lemma}
\begin{proof}
    We have
    \begin{equation}
        \begin{aligned}
            &\left\langle B\psi-B\varphi,\psi-\varphi\right\rangle
            \\
            =&\int_{M}|D\psi|^{p}+|D\varphi|^{p}-|D\psi|^{p-2}\langle D\psi,D\varphi\rangle-|D\varphi|^{p-2}\langle D\varphi, D\psi\rangle dx.
        \end{aligned}
        \nonumber
    \end{equation}
By H$\ddot{o}$lder’s inequality, we have
\begin{equation}
    \begin{aligned}
        &\int_{M}|D\psi|^{p-2}\langle D\psi,D\varphi\rangle dx
        \le&\left(\int_{M}|D\psi|^{p}dx\right)^{\frac{p-1}{p}}\left(\int_{M}|D\varphi|^{p}dx\right)^{\frac{1}{p}}
        \\
            =&\|\psi\|_{1,p}^{p-1}\|\varphi\|_{1,p}
    \end{aligned}
    \nonumber
\end{equation}
Similarly we have
\begin{equation}
    \int_{M}|D\varphi|^{p-2}\langle D\varphi,D\psi\rangle dx\le \left\|\varphi\right\|_{1,p}^{p-1}\left\|\psi\right\|_{1,p}.
    \nonumber
\end{equation}
Therefore,
\begin{equation}
    \begin{aligned}
        \left\langle B\psi-B\varphi,\psi-\varphi\right\rangle
        \ge&\left\|\psi\right\|_{1,p}^{p}+\left\|\varphi\right\|_{1,p}^{p}-\left\|\psi\right\|_{1,p}^{p-1}\left\|\varphi\right\|_{1,p}-\left\|\varphi\right\|_{1,p}^{p-1}\left\|\psi\right\|_{1,p}
        \\
        \ge&\left(\left\|\psi\right\|_{1,p}^{p-1}-\left\|\varphi\right\|_{1,p}^{p-1}\right)\left(\left\|\psi\right\|_{1,p}-\left\|\varphi\right\|_{1,p}\right).
    \end{aligned}
    \nonumber
\end{equation}
Furthermore, it follow that $\langle B\psi-B\varphi,\psi-\varphi\rangle=0$ if and only if $\psi=\varphi$ a.e. in $M$.
\end{proof}

\begin{lemma}\label{lemma2.9}
    Let $G$ be defined as in \eqref{eq2.12}, then $G'$ satisfies $(F3)$.
\end{lemma}

\begin{proof}
Since $B=G'$, it suffices to show this for $B$. Similar to lemma\ref{lemma2.7}, it is easy to verify $B$ is continuous and bounded by Sobolev's embedding theorem and H$\ddot{o}$lder' inequality.

It remains to show that $B$ satifies condition $(S_0)$.  According to  $u_n\rightharpoonup u$ in $H^{1,p}(M,\mathbb{S}(M))$, we have $u_n\rightarrow u$ in $L^p(M,\mathbb{S}(M))$. Since $H^{1,p}(M,\mathbb{S}(M))$ is a reflexive Banach space,
using the Lindenstrauss-Asplund-Troyanski theorem in \cite{Tro} one can find an equivalent norm such that $H^{1,p}(M,\mathbb{S}(M))$ with this norm is locally uniformly convex.
In such a space weak convergence and norm convergence imply convergence. Thus to show
$\psi_{n}\to \psi $ in $H^{1,p}(M,\mathbb{S}(M))$, we only need to show $\left\|\psi_{n}\right\|_{1,p}\to \left\|\psi\right\|_{1,p}$. We have
\begin{equation}
    \lim_{n\to\infty}\left\langle B\psi_{n}-B\psi,\psi_{n}-\psi\right\rangle=\lim_{n\to \infty}\left(\left\langle B\psi_{n}, \psi_{n}\right\rangle-\left\langle B\psi_{n}-\psi\right\rangle-\left\langle B\psi, \psi_{n}-\psi\right\rangle\right)=0
    \nonumber
\end{equation}
On the other hand we have
\begin{equation}
    \left\langle B\psi-B\varphi,\psi-\varphi\right\rangle\ge \left(\left\|\psi\right\|_{1,p}^{p-1}-\left\|\varphi\right\|_{1,p}^{p-1}\right)\left(\left\|\psi\right\|_{1,p}-\left\|\varphi\right\|_{1,p}\right).
    \nonumber
\end{equation}
Hence $\left\|\psi_{n}\right\|_{1,p}\to \left\|\psi\right\|_{1,p}$ as $n\to \infty$, therefore we complete the proof.
\end{proof}
From the theorem\ref{th2.5}, we obtain the following:
\begin{theorem}\label{th2.10}
    There exists a nondecreasing sequence of nonnegative eigenvalues $\left\{\lambda_{n}\right\}$ of \eqref{eq2.9} obtained by using the L-S principle
    such that $\lambda_{n}=\frac{1}{\mu_{n}}-1\to \infty$ as $n\to\infty$,
    where each $\mu_{n}$ is an eigenvalue of the corresponding equation $F'(\psi)=\mu G'(\psi)$.
\end{theorem}
\begin{proof}
    Let $F,G$ be the two functionals defined in \eqref{eq2.11}, \eqref{eq2.12}. Then by the above lemmas,
    there exists a nonincreasing sequence of nonnegative eigenvalues $\mu_{n}$ obtained from the L-S principle such that $\mu_{n}\to 0$
    as $n\to \infty$ where $\mu_{n}=\sup_{H\in A_{n}} \inf_{\psi\in H}F(\psi)$ and each $\mu_{n}$ is an eigenvalue of $F'(\psi)=\mu G'(\psi)$.
    Then it is easy to see \eqref{eq2.9}
    has a nondecreasing sequence of nonnegative eigenvalues $\lambda_{n}=\frac{1}{\mu_{n}}-1\to\infty$ as $n\to \infty$.
\end{proof}

Finally, we can also know easily each eigenvalue $\lambda$ is positive on \eqref{eq2.9}. For every eigenvalue $\lambda$,
by approximation, the corresponding eigenfunction $\psi$ can also be a test-function in \eqref{eq2.9}. Therefore we have
\begin{equation}
    \lambda=\frac{\int_{M}|D\psi|^{p}dx}{\int_{M}|\psi|^{p}dx}
    \nonumber
\end{equation}
and every eigenvalue $\lambda$ is positive. Also eigenvalues have an explicit lower bound.
Using the Sobolev type inequality $\left\|\psi\right\|_{p}\le C\left\|D\psi\right\|_{p}$, where $C=C(n,p)$. We have
\begin{equation}
    \lambda\ge \frac{1}{C}.
    \nonumber
\end{equation}

Remark: There exists a nondecreasing sequence of nonnegative eigenvalues obtained by the Ljusternik-Schnirelman principle,
though it is not known whether there is or not a non Ljusternik-Schnirelman eigenvalue.


\section{p-superlinear case}
In this section, we firstly prove  the Palais-Smale condition for the functional $\mathfrak{L}$  with to $H^{1,p}(M,\mathbb{S}(M))$. Secondly, we use the Mountain Pass Theorem and Fountain Theorem to prove theorem \ref{th1.1} and \ref{th1.2} .

\subsection{The Palais-Smale condition}
In the following, we give the definition of Palais-Smale(denote PS) condition:
\begin{definition}\label{def3.1}
Let  $C^{1} $-functional $L$ be defined on a Banach space $E$. A sequence $\{x_n\}_{n=1}^\infty\subset E$ is called a $(PS)$ sequence if
\begin{enumerate}
\item[$(1)$] $L\left(x_{n}\right)$ is bounded in E;
\item[$(2)$] $\left \| \mathrm{d}L\left ( x_{n} \right ) \right \|_{E^* } \to 0 $ as $n\to \infty$.
\end{enumerate}

If all $(PS)$ sequence contain a convergent subsequence in $E$, we say that $L$ satisfies the $(PS)$ condition with respect to $E$.
\end{definition}

\begin{definition}\label{def3.2}
Let  $C^{1} $-functional $L$ be defined on a Banach space $E$, $c\in \mathbb{R}$. A sequence $\{x_n\}_{n=1}^\infty\subset E$ is called a $(PS)_c$ sequence if
\begin{enumerate}
\item[$(1)$] $L\left ( x_{n} \right )\to c$ as $n\to \infty$;
\item[$(2)$] $\left \| \mathrm{d}L\left ( x_{n} \right ) \right \|_{E^* } \to 0 $ as $n\to \infty$.
\end{enumerate}
If all $(PS)_c$ sequence contain a convergent subsequence in $E$, we say that $L$ satisfies the $(PS)_c$  condition with respect to $E$.
\end{definition}

Remark: It is easy to show that the $(PS)$ condition is equivalent to the  $(PS)_c$ condition for any $c\in \mathbb{R}$.

We then have:
\begin{lemma}\label{lemma3.3}
Suppose $H$ satisfies $\left(H1\right)-\left(H3\right)$,
Then functional $\mathfrak{L}$ satisfies the Palais-Smale condition with to $H^{1,p}(M,\mathbb{S}(M))$.
\end{lemma}

\begin{proof}
 Let $\left ( \psi_{n} \right )\subset H^{1,p}(M,\mathbb{S}(M))$ be a Palais-Smale sequence, $i.e.$,
the following conditions are satisfied:
\begin{equation}\label{eq3.18}
   \mid \mathfrak{L}\left ( \psi _{n}  \right ) \mid\leq C,
\end{equation}
and
\begin{equation}\label{eq3.19}
    \left \| \mathrm{d}\mathfrak{L}  \left ( \psi _{n}  \right )   \right \| _*\to 0 \quad as \ n\to\infty.
\end{equation}
We first show that $\left\{\psi_{n}\right\}\subset H^{1,p}(M,\mathbb{S}(M))$ is bounded.

By \eqref{eq3.19}, for large $n$, for any $\phi\in H^{1,p}(M,\mathbb{S}(M))$, we obtain
$$
    \langle\mathrm{d}\mathfrak{L} \left ( \psi _{n} \right ),\phi\rangle=    \int_{M}(\langle|D\psi_n|^{p-2}D\psi_n,D\phi\rangle-\langle h(x,\psi_n),\phi\rangle)dx=o(1)\|\phi\|_{1,p}.
$$
Let $\phi= \psi _{n}$, we get that
\begin{equation}\label{eq3.20}
    \left \langle \mathrm{d}\mathfrak{L}(\psi_{n}) ,\psi_{n}\right \rangle=\int_{M}(|D\psi_n|^p-\langle h(x,\psi_n),\psi_n\rangle)dx=o(1)\|\psi_n\|_{1,p}.
\end{equation}
Toghter with \eqref{eq3.18}, we can conclude that

\begin{align}\label{eq3.21}
    &C+o(1)\|\psi_n\|_{1,p}
    \ge \mathfrak{L}(\psi_{n})- \frac{1}{\mu}\left \langle \mathrm{d}\mathfrak{L}(\psi_{n}) ,\psi_{n}\right \rangle
    \nonumber\\
    =&\frac{\mu-p}{p\mu}\int_{M} |D\psi_n|^pdx+\int_{\{x\in M:|\psi_n(x)|\geq R_1\}}(\frac{1}{\mu}\langle h(x,\psi_n),\psi_n\rangle-H(x,\psi_n))dx
    \nonumber\\
    +&\int_{\{x\in M:|\psi_n(x)|< R_1\}}(\frac{1}{\mu}\langle h(x,\psi_n),\psi_n\rangle-H(x,\psi_n))dx
\end{align}
Using condition $(H1)$, there exist $R_1>0$ and $C_4>0$ such that $\frac{1}{\mu}\langle h(x,\psi_n),\psi_n\rangle-H(x,\psi_n)\geq-C_4$
with $|\psi_n|<R_1$. Together \eqref{eq3.21} and condition $(H2)$, we then have
$$
\frac{\mu-p}{p\mu}\int_{M} |D\psi_n|^pdx\leq C+o(1)\|\psi_n\|_{1,p}+C_4.
$$
This prove that $\left\{\psi_{n}\right\}\subset H^{1,p}(M,\mathbb{S}(M))$ is bounded.

Since $\left\{\psi_{n}\right\}\subset H^{1,p}(M,\mathbb{S}(M))$ is bounded, after passing to a subsequence,
we may assume that there exists $\psi\in H^{1,p}(M,\mathbb{S}(M))$ such that $\psi_{n}\rightharpoonup \psi$ weakly in $H^{1,p}(M,\mathbb{S}(M))$
and $\psi_{n}\to \psi$ strongly in $L^{r}(M)$ for any $1\le r<\frac{mp}{m-p}$.

By \eqref{eq3.19}
and the boundedness of $\left\{\psi_{n}\right\}$ in $H^{1,p}(M,\mathbb{S}(M))$, we have
\begin{align}
    o(1)=&\left\langle\mathrm{d}\mathfrak{L}(\psi_{n}),\psi_{n}-\psi\right\rangle
    \nonumber\\\label{eq3.22}
    =&\int_{M}\left\langle\psi_{n}-\psi,D_{p}\psi_{n}\right\rangle dx-\int_{M}\left\langle\psi_{n}-\psi,H_{\psi}(x,\psi_{n})\right\rangle dx
\end{align}
Here, by $(H1)$ we have
\begin{align}
    \left| \int_{M}\left\langle\psi_{n}-\psi,H_{\psi}(x,\psi_{n})\right\rangle dx\right|
    \le&C_1\int_{M}(1+|\psi_{n}|^{q-1})|\psi_{n}-\psi|dx
    \nonumber\\
    \le&C_5\left(\left\|\psi_{n}-\psi\right\|_{1}+\left\|\psi_{n}\right\|_{q}^{q-1}\left\|\psi_{n}-\psi\right\|_{q}\right)
    \nonumber\\\label{eq3.23}
    =&o(1)
\end{align}
as $n\to\infty$.

By \eqref{eq3.22} and \eqref{eq3.23}, we have
\begin{equation}\label{eq3.24}
    \left\|\psi_{n}-\psi\right\|_{1,p}^{p}=\int_{M}\left\langle\psi_{n}-\psi,D_{p}(\psi_{n}-\psi)\right\rangle\mathrm{d}x=o(1)
\end{equation}
as $n\to \infty$.

Hence, we have $\psi_{n}\to \psi$ in $H^{1,p}(M,\mathbb{S}(M))$. This completes the verification of the Palais-Smale condition.
\end{proof}

\subsection{proof of Theorem \ref{th1.1}}

The proof of Theorem \ref{th1.1} in the introduction is based on an application of the following theorem:

\begin{theorem}\label{th3.4}
(Mountain Pass Theorem). Let $E$ be a Banach space and  $I\in C^1(E,\mathbb{R})$. Suppose that
\begin{flalign*}
(I_1)&\text{there exist}\ r>0\ and\  \rho>0\ \text{such that}\ I(u)\geq \rho\ if\ u\in S_r=\{u\in E:\|u\|=r\};&\\
(I_2)&\text{there exists}\ e\in E\ \text{with}\ \|e\|>r\ \text{such that}\ I(e)\leq 0;&
\end{flalign*}
we set
\begin{equation}
        c=\inf_{\gamma\in \Gamma} \max_{t\in[0,1]} I(\gamma(t)).
        \nonumber
\end{equation}
where $\Gamma=\{\gamma\in C([0,1],E):\gamma(0)=0,\gamma(1)=e\}.$

Then $I$ have a $(PS)_c$ sequence. Moreover, if $I$ satisfy $(PS)_c$ condition,
then $I$ possesses a critical value $c$.
\end{theorem}
First step: we give the condition $(I_1)$ of Mountain Pass Theorem in the following lemma:
\begin{lemma}\label{lemma3.5}
Assume that $H$ satisfies $(H1)$ and $(H3)$ in the introduction. Then there exist $r, \rho>0$ such that $I(\psi)\le0$ with $\psi\in S_r=\{\psi\in H^{1,p}(M,\mathbb{S}(M)):\|\psi\|_{1,p}=r\}$.
\end{lemma}

\begin{proof}
Using the condition $(H1)$ and $(H3)$, for any $\varepsilon>0$, there exist $C_\varepsilon>0$ such that
$$
H(x,\psi)\leq\varepsilon|\psi|^p+C_\varepsilon|\psi|^q.
$$
Therefore, by the Sobolev embedding theorem and H$\ddot{o}$lder inequality, we get that
\begin{align}\label{eq3.25}
        \mathfrak{L}(\psi)\geq&\frac{1}{p}\int_{M}|D\psi|^pdx-\varepsilon\int_{M}|\psi|^pdx-C_\varepsilon\int_M|\psi|^qdx
     \nonumber\\
     \ge&\frac{1}{p}\left\|\psi\right\|_{1,p}^{p}-\varepsilon C_{5}\left\|\psi\right\|_{1,p}^{p} -C_\varepsilon C_6\left\|\psi\right\|_{1,p}^{q}
     \nonumber\\
     \ge& C_{7}\left\|\psi\right\|_{1,p}^{p}(1-C_{8}\left\|\psi\right\|_{1,p}^{q-p}).
    \end{align}
We choose $\left\|\psi\right\|_{1,p}:=r<(\frac{1}{2C_8})^{\frac{1}{q-p}}$ and use  \eqref{eq3.25} to lead to
$$
\mathfrak{L}(\psi)\geq \frac{C_{7}}{2}\left\|\psi\right\|_{1,p}^{p}>0
$$
We set $\rho=\frac{C_{7}}{2}\left\|\psi\right\|_{1,p}^{p}$, and this concludes the proof.
\end{proof}

Second step: we give the condition $(I_2)$ of Mountain Pass Theorem in the following lemma:

\begin{lemma}\label{lemma3.6}
    Assume that $H$ satisfies $(H2)$ in the introduction.  Then $ \mathfrak{L}$ satisfy the condition $(I_2)$ of Mountain Pass Theorem.
\end{lemma}
\begin{proof}
    By $(H2)$, there exists $C_9,C_{10}>0$ such that
    \begin{equation}\label{eq3.26}
       H(x,\psi)\geq C_9|\psi|^\mu -C_{10}
    \end{equation}
    for any $(x,\psi)\in \mathbb{S}(M)$.

    We take $e_0\in E$ and $|e_0|\neq 0$, which combines with \eqref{eq3.26}, we have
    \begin{align}
        \mathfrak{L}(te_0)=&\frac{1}{p}t^p\int_{M}|De_0|^pdx-\int_{M}H(x,te_0)dx
     \nonumber\\
     \leq &\frac{1}{p}t^p\int_{M}|De_0|^pdx-C_9t^\mu\int_M|e_0|^\mu dx+C_{10}.
    \end{align}
   It concludes  $\mathfrak{L}(te_0)\rightarrow -\infty$ as $t\rightarrow \infty$.

   Hence for $r$ in lemma \ref{lemma3.5}, there exist $t_0>0$ such that $ \mathfrak{L}(t_0e_0)\leq 0$ and $\|t_0e_0\|_{1,p}>r$.
Set $e=t_0e_0$, this completes the verification of the condition $(I_2)$ of Mountain Pass Theorem.
\end{proof}
We are now in the position to complete the proof of Theorem 1.1.

\textbf{ Proof of Theorem 1.1:}
Using lemma \ref{lemma3.3}, lemma \ref{lemma3.5} and lemma \ref{lemma3.6}, $\mathfrak{L}$ satisfy Mountain Pass Theorem. The proof of Theorem 1.1 is completed.

\subsection{Proof of Theorem \ref{th1.2}}
 We first recall the Founntain theorem, see \cite{Wil} for the detailed exposition.

Let $X$ be a Banach space with basis $\left\{e_{j}\right\}_{j=1}^{\infty}$, $i.e.$, $X=\overline{span \left\{e_{j}\right\}_{j=1}^{\infty}}$.
For $k\ge 1$, we set $Y_k=span\left\{e_{j}\right\}_{j=1}^{k}$ and $Z_k=\overline{span \left\{e_{j}\right\}_{j=k}^{\infty}}$.

The fountain theorem is stated as follows:
\begin{theorem}\label{th3.7}
    (Fountain theorem) Let $X=Y_k+Z_k$ be as above. suppose  $L$ satisfy $(PS)$ condition and $L\in C^{1}(H,\mathbb{R})$ is an even functional, $i.e.$, $L(-u)=L(u)$ for all $u\in H$. If for every $k\in \mathbb{N}$, there exists $\rho_k>r_k>0$ such that
 \begin{flalign*}
(L_1)&a(k):=\max \left\{\mathfrak{L}(\psi)|\psi\in Y_k,\left\|\psi\right\|_X=\rho_k\right\}\leq 0;&\\
(L_2)&b(k):=\inf \left\{\mathfrak{L}(\psi)|\psi\in Z_k,\left\|\psi\right\|_X=r_k\right\}\to\infty,\quad k\to +\infty.;&
\end{flalign*}
Then $L$ possesses an unbounded sequence of critical values.
\end{theorem}

Since $H^{1,p}(M,\mathbb{S}(M))$ is not a Hilbert space, We introduce the theory of biorthogonal system in (\cite{LiTz},p.42-43).
\begin{definition}\label{def3.8}
Let $E$ be a Banach space. $E^*$ is the dual space of $E$. A pair of sequences $\{x_n\}_{n=1}^\infty$ in $E$ and $\{x^*_n\}_{n=1}^\infty$ in $E^*$
is called a  biorthogonal system if $x^*_m(x_n)=\delta_n^m$. A sequence $\{x_n\}_{n=1}^\infty$ in $E$ is called a minimal system if there exists a sequence  $\{x^*_n\}_{n=1}^\infty$ in $E^*$ such that $(\{x_n\}_{n=1}^\infty,\{x^*_n\}_{n=1}^\infty)$ is a biorthogonal system.
\end{definition}
\begin{definition}\label{def3.9}
A minimal system $\{x_n\}_{n=1}^\infty$ in $E$ is called fundamental if $x^*(x_n)=0$ implies $x^*=0$ for all $n$; A minimal system $\{x_n^*\}_{n=1}^\infty$ in $E$ is called total if $x_n^*(x)=0$ implies $x=0$ for all $n$.
\end{definition}
In the following, we give the existence of biorthogonal system on separable Banach space:
\begin{proposition}\label{pro3.10}
Let $E$ be a separable Banach space. Then $E$ contains a fundamental minimal system whose biothogonal functionals are total.
\end{proposition}
Nowing using the proposition \ref{pro3.10}, since $H^{1,p}:=H^{1,p}(M,\mathbb{S}(M))$ is a separable Banach space, there exists a biorthogonal system $(e_m,e_n^*)_{m,n\in \mathbb{N}}$ such that $e_m$ in $H^{1,p}$, $e_n^*$ in $(H^{1,p})^*$, the $e_m$'s are linearly dense  in $H^{1,p}$ and the $e_n^*$'s  are total for $H^{1,p}$. Let us remark that we could in particular choose $\{e_n\}_{n=1}^\infty$ to be a Schauder basis for $H^{1,p}$, and then find biorthogonal functionals $\{e^*_n\}_{n=1}^\infty$. Denote
$$
H_k=span\{e_1,e_2,...,e_k\},\quad H_k^\perp=\overline{span\{e_{k+1},e_{k+2},...\}},
$$
and
$$
[e_1^*,e_2^*,...,e_k^*]^{\perp}=\{\psi\in H^{1,p}(M,\mathbb{S}(M))|e_1^*(\psi)=\cdot\cdot\cdot=e_k^*(\psi)=0\}.
$$
It is clear that
$$
H^{1,p}(M,\mathbb{S}(M))=H_k\oplus H_k^\perp,\quad H_k^\perp\subset [e_1^*,e_2^*,...,e_k^*]^{\perp}.
$$
First of all, we prove:

\begin{lemma}\label{lemma3.11}
    Assume $(H2)$ is satisfied. For each $k\ge 1$, there exists $\rho(k)>0$ large enough such that $\mathfrak{L}(\psi)\le 0$ provided $\psi \in E_k$ with $\left\|\psi\right\|_{1,p}\ge \rho(k)$.
\end{lemma}

\begin{proof}
    Using $(H2)$ and  \eqref{eq3.26}, we have
    \begin{align}\label{eq3.28}
        \mathfrak{L}(\psi)\le &\frac{1}{p}\int_{M}\left\langle\psi,D_{p}\psi\right\rangle dx-C_{9}\int_{M}|\psi|^{\mu}dx+C_{10}
        \nonumber\\
        = &\frac{1}{p}\left\|\psi\right\|_{1,p}^{p}-C_9\int_{M}|\psi|^{\mu}dx+C_{10}.
    \end{align}

    Since $\dim H_k= k<\infty$, the $L^{\mu}$ and the $H^{1,p}-$norms are equivalent on $H_k$, hence there exists a  $C(k)>0$ (depending only on $k$) such that
     $\left\|\psi\right\|_{1,p}\le C(k)\left\|\psi\right\|_{\mu}$ hold for any $\psi\in H_k$. By \eqref{eq3.28},
     we can conclude
     \begin{equation}\label{eq3.29}
         \mathfrak{L}(\psi) \le \frac{1}{p}\left\|\psi\right\|_{1,p}^{p}-C(k)\left\|\psi\right\|_{1,p}^{\mu}+C_{10}
     \end{equation}
     for any $\psi\in H_k$, where $C(k)>0$ is an another constant (we use the same notation for simplicity) depending only on $k$.

     From \eqref{eq3.29}, since $\mu>p$, we easily see that there exists $\rho(k)>0$ large enough such that $\mathfrak{L}(\psi)\le 0$
     provided $\psi\in H_k$ and $\left\|\psi\right\|_{1,p}\ge \rho(k)$.
\end{proof}

To proceed further, we need the follows:

\begin{lemma}\label{lemma3.12}
Let $E'$ be a Banach space, $E$ be a separable Banach space, The embedding $E\hookrightarrow E'$ is compact. Let $\tau_m=\inf_{u\in K_m}\|u\|_E$, where $K_m=\{u\in E|\|u\|_{E'}=1,u\in [x_1^*,x_2^*,...,x_m^*]^{\perp}\}$.
 We then have
 \begin{equation}
    \|u\|_{E'}\leq \tau_m^{-1}\|u\|_E,\quad \text{for any}\  u\in [x_1^*,x_2^*,...,x_m^*]^{\perp}\nonumber
 \end{equation}
 and
 $\tau_m\to +\infty$ as $m\to \infty$.
\end{lemma}
\begin{proof}
We only need to  prove $\tau_m\to +\infty$ as $m\to \infty$. Assume $\tau_m$ is bounded, then there exists a sequence $\{u_m\}$ with $\|u_m\|_{E'}=1$ and $u\in [x_1^*,x_2^*,...,x_m^*]^{\perp}$ such that $\|u_m\|_E\leq C$.

Since $E$ is reflexive, after passing to a subsequence, we may assume that there exists $u\in E$ such that $u_m\rightharpoonup u$ weakly in $E$. Using the embedding $E\hookrightarrow E'$ is compact, we have $u_m\rightarrow u$  in $E'$.

By Mazur theorem, if  $u_m\rightharpoonup u$ weakly in $E$, then we have
$$
\sum_{j=m}^{k_m}\alpha_j^{(m)}u_j\rightarrow u,\quad \text{in}\ E,
$$
where $\alpha_j^{(m)}\geq 0,\sum_{j=m}^{k_m}\alpha_j^{(m)}=1$, $m$ is an arbitrary positive integer.

for any $i$, we have
\begin{equation}\label{eq3.30}
x_i^*(u)=\lim_{m\rightarrow \infty}x_i^*(\sum_{j=m}^{k_m}\alpha_j^{(m)}u_j)=\lim_{m\rightarrow \infty}\sum_{j=m}^{k_m}\alpha_j^{(m)}x_i^*(u_j)=0.
\end{equation}
Since $\{x_n^*\}_{n=1}^\infty$ in $E$ is total, from \eqref{eq3.30}, we have $u=0$. However  we have been $\|u\|_{E'}=1$ by $\|u_m\|_{E'}=1$
and $u_m\rightarrow u$  in $E'$. This leads to contradiction.
\end{proof}

From Lemma \ref{lemma3.12}, we obtain the following:

\begin{lemma}\label{lemma3.13}
    Suppose $(H1)$  and $(H3)$ is satisfied. There exists $r(k)>0$ such that
    \begin{equation}
        b(k):=\inf \left\{\mathfrak{L}(\psi)|\psi\in H_k^\perp,\left\|\psi\right\|_{1,p}=r(k)\right\}\to+\infty
        \nonumber
    \end{equation}
    as $k\to \infty$.
\end{lemma}

\begin{proof}
Using the condition $(H1)$ and $(H3)$, for any $\varepsilon>0$, there exist $C_\varepsilon>0$ such that
\begin{equation}\label{eq3.31}
  H(x,\psi)\leq\varepsilon|\psi|^p+C_\varepsilon|\psi|^q.
\end{equation}

Therefore,  think of $L^p(M,\mathbb{S}(M))$ and $L^q(M,\mathbb{S}(M))$ as $E'$ in lemma\ref{lemma3.12}, by \eqref{eq3.31}, we have for $\psi \in H_k^\perp$ with $\left\|\psi\right\|_{1,p}=r:=r(k)$
    \begin{align}\label{eq3.32}
        \mathfrak{L}(\psi)=&\frac{1}{p}\int_{M}\left\langle\psi,D_{p}\psi\right\rangle dx-\int_{M}H(x,\psi)dx
        \nonumber\\
       \geq&\frac{1}{p}\int_{M}|D\psi|^pdx-\varepsilon\int_{M}|\psi|^pdx-C_\varepsilon\int_M|\psi|^qdx-C
        \nonumber\\
        \ge&\frac{1}{p}\|\psi\|_{1,p}^{p}-\varepsilon \tau_k^{-p} \|\psi\|_{1,p}^{p}-C_\varepsilon\tau_k^{-q} \|\psi\|_{1,p}^{q}-C
         \nonumber\\
        \ge&\frac{1}{p}r^{p}-\varepsilon \tau_k^{-p} r^{p}-C_\varepsilon\tau_k^{-q} r^{q}-C,
    \end{align}

   Again by Lemma \ref{lemma3.12}, we have $\varepsilon \tau_k^{-p}\le \frac{1}{2p}$ for large $k$ , we have from \eqref{eq3.32}
 \begin{equation}\label{eq3.33}
  \mathfrak{L}(\psi) \ge\frac{1}{2p}r^{p}-C_\varepsilon\tau_k^{-q} r^{q}-C.
\end{equation}
Hence
    \begin{equation}\label{eq3.34}
        b(k)=\inf \left\{\mathfrak{L}(\psi)|\psi\in H_k^\perp,\left\|\psi\right\|_{1,p}=r(k)\right\}\ge \frac{1}{2p}r^{p}-C_\varepsilon\tau_k^{-q} r^{q}-C.
    \end{equation}
    Since the maximum of the function $r\mapsto \frac{1}{2p}r^{p}-C_\varepsilon\tau_k^{-q} r^{q}$ is $\frac{q-p}{2pq}\left(2qC_\varepsilon\tau_k^{-q}\right)^\frac{p}{p-q}$ and is
    attained at $r=\left(2qC_\varepsilon\tau_k^{-q}\right)^\frac{1}{p-q}$, we take $r=\left(2qC_\varepsilon\tau_k^{-q}\right)^\frac{1}{p-q}$ and since $p<q$, we obtain

    \begin{align}
         b(k)=\inf \left\{\mathfrak{L}(\psi)|\psi\in H_k^\perp,\left\|\psi\right\|_{1,p}=r(k)\right\}
        \ge\frac{q-p}{2pq}\left(2qC_\varepsilon\tau_k^{-q}\right)^\frac{p}{p-q}-C\to +\infty
        \nonumber
    \end{align}
    as $k\to \infty$. This completes the proof.
\end{proof}

We are now in a position to complete the proof of Theorem 1.2.

\textbf{Proof of theorem 1.2:}
    Taking $\rho(k)$ large if necessary, we have $r(k)<\rho(k)$ for each $k$. By Lemmas\ref{lemma3.3}, Lemmas\ref{lemma3.11},
      Lemmas\ref{lemma3.13}, theorem \ref{th1.2} holds. Hence the existence of a sequence ${\psi_n}\subset H^{1,p}(M,\mathbb{S}(M))$ satisfying
$$
\mathfrak{L}(\psi_n)\rightarrow c_k,\quad \mathfrak{L}'(\psi_n)\rightarrow 0,
$$
where $c_k$ is a critical value of $\mathfrak{L}$. Since $c_k\geq b_k$ and $b_k\rightarrow +\infty$ as $k\rightarrow \infty$.
The proof is completed.

\section{p-sublinear case}
\subsection{The Palais-Smale condition}

First of all, we present the following special case of Minimax theorem:
\begin{theorem}\label{th4.1}
 Let $E$ be a Banach space and  $I\in C^1(E,\mathbb{R})$. If $I$ satisfy $(PS)$ condition and $I$ is bounded below on $E$,
then
$$
c:=\inf\{I(x):x\in E\}
$$
 is a critical value of $I$.
\end{theorem}

In the following, we use a theorem to prove the theorem 1.3. Firstly, we give the Palais-Smale condition under p-sublinear case:
\begin{lemma}\label{lemma4.2}
Suppose $H$ satisfies $(Hi)$.
Then functional $\mathfrak{L}$ satisfies the Palais-Smale condition with to $H^{1,p}(M,\mathbb{S}(M))$.
\end{lemma}

\begin{proof}
Using the condition $(Hi)$, there exist $C_{11},C_{12}>0$ such that
\begin{equation}\label{eq4.35}
  H(x,\psi)\leq C_{11}|\psi|^q+C_{12}.
\end{equation}
Therefore, by the Sobolev embedding theorem and H$\ddot{o}$lder inequality, we get that
\begin{align}\label{eq4.36}
        \mathfrak{L}(\psi)\geq&\frac{1}{p}\int_{M}|D\psi|^pdx-C_{11}\int_M|\psi|^qdx-C
     \nonumber\\
     \ge&\frac{1}{p}\left\|\psi\right\|_{1,p}^{p} -C_{11}\left\|\psi\right\|_{1,p}^{q}-C.
    \end{align}
 Since $1<q<p$,  we have that $\mathfrak{L}(\psi)\to \infty$ as $\|\psi\| \to \infty$. This proves that $\mathfrak{L}$ is coercive.

 Let $\{ \psi_{n}\}\subset H^{1,p}(M,\mathbb{S}(M))$ be a Palais-Smale sequence,i.e. $\{ \psi_{n}\}$ satisfied:
\begin{equation}\label{eq4.37}
   \mid \mathfrak{L}\left ( \psi _{n}  \right ) \mid\leq C,
\end{equation}
and
\begin{equation}\label{eq4.38}
    \left \| \mathrm{d}\mathfrak{L}  \left ( \psi _{n}  \right )   \right \| _*\to 0 \quad as \ n\to\infty.
\end{equation}
 By the coerciveness of $\mathfrak{L}$, we conclude that $\{\psi _{n}\}$ is bounded in $H^{1,p}(M,\mathbb{S}(M))$.

Using $\left\{\psi_{n}\right\}\subset H^{1,p}(M,\mathbb{S}(M))$ is bounded, after passing to a subsequence,
we may assume that there exists $\psi\in H^{1,p}(M,\mathbb{S}(M))$ such that $\psi_{n}\rightharpoonup \psi$ weakly in $H^{1,p}(M,\mathbb{S}(M))$
and $\psi_{n}\to \psi$ strongly in $L^{r}(M)$ for any $1\le r<\frac{mp}{m-p}$.

By \eqref{eq4.38}
and the boundedness of $\left\{\psi_{n}\right\}$ in $H^{1,p}(M,\mathbb{S}(M))$, we have
\begin{align}\label{eq4.39}
    o(1)=&\left\langle\mathrm{d}\mathfrak{L}(\psi_{n}),\psi_{n}-\psi\right\rangle
    \nonumber\\
    =&\int_{M}\left\langle\psi_{n}-\psi,D_{p}\psi_{n}\right\rangle dx-\int_{M}\left\langle\psi_{n}-\psi,H_{\psi}(x,\psi_{n})\right\rangle dx
\end{align}
Here, by $(Hi)$ we have
\begin{align}\label{eq4.40}
    \left| \int_{M}\left\langle\psi_{n}-\psi,H_{\psi}(x,\psi_{n})\right\rangle\mathrm{d}x\right|
    \le&C_2\int_{M}(1+|\psi_{n}|^{q-1})|\psi_{n}-\psi|dx
    \nonumber\\
    \le&C\left(\left\|\psi_{n}-\psi\right\|_{1}+\left\|\psi_{n}\right\|_{q}^{q-1}\left\|\psi_{n}-\psi\right\|_{q}\right)
    \nonumber\\
    =&o(1)
\end{align}
as $n\to\infty$.

By \eqref{eq4.39} and \eqref{eq4.40}, we have
\begin{equation}\label{eq4.41}
    \left\|\psi_{n}-\psi\right\|_{1,p}^{p}=\int_{M}\left\langle\psi_{n}-\psi,D_{p}(\psi_{n}-\psi)\right\rangle dx=o(1)
\end{equation}
as $n\to \infty$.

Hence, we have $\psi_{n}\to \psi$ in $H^{1,p}(M,\mathbb{S}(M))$. This completes the verification of the Palais-Smale condition.
\end{proof}

\subsection{Proof of Theorem \ref{th1.3}}
We are now in a position to complete the proof of Theorem 1.3.

\textbf{Proof of theorem 1.3:}
   By \eqref{eq4.36}, let function $f(r):=\frac{1}{p}r^{p} -Cr^{q}-C$. The minimum of the function $f(r)=\frac{1}{p}r^{p} -Cr^{q}-C$ is $\frac{q-p}{pq}(Cq)^{\frac{p}{p-q}}-C$ and is
    attained at $r=(Cq)^\frac{1}{p-q}$. This proves that $\mathfrak{L}$ is bounded below.

    Hence, according to lemma \ref{lemma4.2} and theorem \ref{th4.1}, $c:=\inf\{\mathfrak{L}(\psi):\psi\in H^{1,p}(M,\mathbb{S}(M))\}$ is a critical value of $\mathfrak{L}$, i.e., there exists a point $\psi_0\in H^{1,p}(M,\mathbb{S}(M))$ such that $d\mathfrak{L}(\psi_0)=0, \mathfrak{L}(\psi_0)=c$.
The proof is completed.

\subsection{Proof of Theorem \ref{th1.4}}
 We first recall the following critical point theorem:

Let $X$ be a Banach space with basis $\left\{e_{j}\right\}_{j=1}^{\infty}$, $i.e.$, $X=\overline{span \left\{e_{j}\right\}_{j=1}^{\infty}}$.
For $k\ge 1$, we set $Y_k=span\left\{e_{j}\right\}_{j=1}^{k}$ and $Z_k=\overline{span \left\{e_{j}\right\}_{j=k}^{\infty}}$.

The critical point theorem which is probably well-know is stated as follows:
\begin{theorem}\label{th4.3}
     Let $X=Y_k\oplus Z_{k+1}$ be as above. suppose $L\in C^{1}(X,\mathbb{R})$ is an even functional, $i.e.$, $L(-u)=L(u)$ for all $u\in X$ and
satisfies the following:
 \begin{flalign*}
(L_3)&\text{there exist}\ r(k)>0\ and\  a(k)\in \mathbb{R}\ \text{such that}\ \sup\{L(u):u\in Y_k,\|u\|=r(k)\}\leq a(k);&\\
(L_4)&\text{there exists}\ b(k)\in \mathbb{R}\  \text{such that}\ \inf\{L(u):u\in \mathbb{R}e_k\oplus Z_{k+1}\}\geq b(k).&
\end{flalign*}
Then we have(obviously) $b(k)\leq a(k)$. Assume further that $L$ satisfies $(PS_{c(k)})$ condition for any $c(k)\in[b(k),a(k)]$, then there exist critical  values of $L$ in $[b(k),a(k)]$.
\end{theorem}

Nowing using the proposition \ref{pro3.10}, we set
$$
H^{1,p}(M,\mathbb{S}(M))=H_k\oplus H_k^\perp,\quad H_k^\perp\subset [e_1^*,e_2^*,...,e_k^*]^{\perp}.
$$
where

$$
H_k=span\{e_1,e_2,...,e_k\},\quad H_k^\perp=\overline{span\{e_{k+1},e_{k+2},...\}},
$$
and
$$
[e_1^*,e_2^*,...,e_k^*]^{\perp}=\{\psi\in H^{1,p}(M,\mathbb{S}(M))|e_1^*(\psi)=\cdot\cdot\cdot=e_k^*(\psi)=0\}.
$$

First of all, we prove:

\begin{lemma}\label{lemma4.4}
    Assume $(Hii)$ is satisfied. For any $k\geq1$, there exist $r(k)>0$ and $a(k)<0$ such that $\mathfrak{L}(\psi)\le a(k)$ for any $\psi\in H_k$.
\end{lemma}

\begin{proof}
    Using $(Hii)$, we have
    \begin{align}\label{eq4.42}
        \mathfrak{L}(\psi)\le &\frac{1}{p}\int_{M}\left\langle\psi,D_{p}\psi\right\rangle dx-C_3\int_{M}|\psi|^{\nu}dx
        \nonumber\\
         =&\frac{1}{p}\left\|\psi\right\|_{1,p}^{p}-C_3\int_{M}|\psi|^{\nu}dx.
    \end{align}

    Since $\dim H_k= k<\infty$, the $L^{\mu}$ and the $H^{1,p}-$norms are equivalent on $H_k$, hence there exists a  $C(k)>0$ (depending only on $k$) such that
     $C(k)\left\|\psi\right\|_{1,p}^\nu\le C_3\left\|\psi\right\|_{\nu}^\nu$ hold for any $\psi\in H_k$. By \eqref{eq4.42},
     we can conclude
     \begin{equation}\label{eq4.41}
         \mathfrak{L}(\psi) \le \frac{1}{p}\left\|\psi\right\|_{1,p}^{p}-C(k)\left\|\psi\right\|_{1,p}^{\nu}
     \end{equation}
     for any $\psi\in H_k$.

     From \eqref{eq4.41}, since $1<\nu<p$, we easily see that there exists $r(k)>0$ small enough such that $\mathfrak{L}(\psi)\le\frac{1}{p}\left\|r(k)\right\|_{1,p}^{p}-C(k)\left\|r(k)\right\|_{1,p}^{\nu} <0$
     for any  $\psi\in H_k$ and $\left\|\psi\right\|_{1,p}=r(k)$. We set $a(k)=\frac{1}{p}\left\|r(k)\right\|_{1,p}^{p}-C(k)\left\|r(k)\right\|_{1,p}^{\nu}<0$, This completes the assertion of the lemma.
\end{proof}

To proceed further, we need the follows:

\begin{lemma}\label{lemma4.5}
    Suppose $(Hi)$ is satisfied. Then there exists $b(k)<0$ such that  $b(k)\to 0$ as $k\to \infty$ and
$\inf\{L(\psi):\psi\in \mathbb{R}e_k\oplus Z_{k+1}\}\geq b(k)$.
\end{lemma}

\begin{proof}
Using the condition $(Hi)$, there exist $C>0$ such that
\begin{equation}\label{eq4.44}
  H(x,\psi)\leq C|\psi|+C|\psi|^q.
\end{equation}

Therefore,  think of $L^q(M,\mathbb{S}(M))$ as $E'$ in lemma \ref{lemma3.12}, by \eqref{eq4.44}, we have for $\psi \in H_{k-1}^\perp$
    \begin{align}\label{eq4.45}
        \mathfrak{L}(\psi)=&\frac{1}{p}\int_{M}\left\langle\psi,D_{p}\psi\right\rangle dx-\int_{M}H(x,\psi)dx
        \nonumber\\
       \geq&\frac{1}{p}\int_{M}|D\psi|^pdx-C\int_M|\psi|^qdx-C\int_M|\psi|dx
        \nonumber\\
        \ge&\frac{1}{p}\|\psi\|_{1,p}^{p}-C\tau_{k,q}^{-q} \|\psi\|_{1,p}^{q}-C\tau_{k,1}^{-1} \|\psi\|_{1,p}
        \nonumber\\
        =&\left(\frac{1}{2p}\|\psi\|_{1,p}^{p}-C\tau_{k,q}^{-q} \|\psi\|_{1,p}^{q}\right)+\left(\frac{1}{2p}\|\psi\|_{1,p}^{p}-C\tau_{k,1}^{-1} \|\psi\|_{1,p}\right).
    \end{align}

    Since the minimum of the function $s\mapsto \frac{1}{2p}s^{p}-C\tau_{k,q}^{-q}s^{q}$ is $\frac{q-p}{2pq} (2qC\tau_{k,q}^{-q})^\frac{p}{p-q}$ and is
    attained at $s=(2qC\tau_{k,q}^{-q})^\frac{1}{p-q}$.
    The minimum of the function $s\mapsto \frac{1}{2p}s^{p}-C\tau_{k,1}^{-1}s$ is $\frac{1-p}{2p} (2C\tau_{k,1}^{-1})^\frac{p}{p-1}$ and is
    attained at $s=(2C\tau_{k,1}^{-1})^\frac{1}{p-1}$.

Let $b(k)=\frac{q-p}{2pq} (2qC\tau_{k,q}^{-q})^\frac{p}{p-q}+\frac{1-p}{2p} (2C\tau_{k,1}^{-1})^\frac{p}{p-1}$.
     Since $1<q<p$, by lemma \ref{lemma3.12}, we obtain $b(k)<0$ and $b(k)\to 0$
    as $k\to \infty$. This completes the proof.
\end{proof}

We are now in a position to complete the proof of Theorem 1.4.

\textbf{Proof of theorem 1.4: }
     By theorem \ref{th4.3}, lemma \ref{lemma4.2}, lemma \ref{lemma4.4}, lemma \ref{lemma4.5}. For each $k\geq1$ and $b(k)\leq c_k\leq a(k) $, there exist  critical points sequence $\{\psi_k\}\subset H^{1,p}(M,\mathbb{S}(M))$ (after taking a subsequence if necessary) satisfying
$$
\mathfrak{L}(\psi_k)= c_k,\quad \mathfrak{L}'(\psi_k)= 0,
$$
where $c_k$ is a critical value of $\mathfrak{L}$. Since $b(k)\leq c_k\leq a(k) $ and $b(k)\to 0$ as $k\rightarrow \infty$.
The proof is completed.


\newpage

    \renewcommand{\thesection}{\Alph{section}}
 \section{}
\begin{proposition}\label{prop:A.1.}
Assume that  $H\in C^0(M,\mathbb{S}(M))$ is $C^1$ in the fiber directioni.e., $H_\psi(x,\psi)\in C^0(M,\mathbb{S}(M))$, and there exist $q\in (1,p^*)$ and $C>0$ such that
    \begin{equation}\label{eq:A.1}
        |H_{\psi}(x,\psi)|\le C\left(1+|\psi|^{q-1}\right)
    \end{equation}
for any $(x,\psi)\in \mathbb{S}(M)$.
    Then functional  $\mathcal{H}:H^{1,p}(M,\mathbb{S}(M))\to \mathbb{R}$
 defined by
 \begin{equation} \label{eq:A.2}
\mathcal{H}(\psi)=\int_MH(x,\psi)dx,
\end{equation}
 is of class $C^1$, and  at each $\psi\in H^{1,p}(M,\mathbb{S}(M))$, (Fr$\acute{e}$chet) derivations $\mathcal{H}^\prime(\psi)$
 is given by
\begin{equation} \label{eq:A.3}
\mathcal{H}^\prime(\psi)\xi=\int_M\langle H_\psi(x,\psi),\xi\rangle dx\quad\forall \xi\in H^{1,p}(M,\mathbb{S}(M)).
\end{equation}
\end{proposition}

\noindent{\bf Proof of Proposition~\ref{prop:A.1.}}.\quad

\noindent{\bf Step 1.}\; {\it $\mathcal{H}$ is G\^ateaux differentiable}.

Given $\psi, \xi\in H^{1,p}(M,\mathbb{S}(M))$ and $t\in (-1,1)\setminus\{0\}$ we have, by the classical Mean Value Theorem,
\begin{eqnarray*}
\frac{H(x,\psi+t\xi)-H(x,\psi)}{t}
=\langle H_\psi(x,\psi+\theta\xi),\xi\rangle
\end{eqnarray*}
for some $\theta=\theta(t,x, \psi, \xi)\in (0,1)$, and thus (\eqref{eq:A.1}) yields
\begin{eqnarray*}
\left|\frac{H(x,\psi+t\xi)-H(x,\psi)}{t}\right|
&\le& C\left(1+|\psi+\theta\xi|^{q-1}\right)|\xi|\\
&\le& C\left(1+ 2^{q-1}|\psi|^{q-1}+ 2^{q-1}|\xi|^{q-1} \right)|\xi|.
\end{eqnarray*}
Recalling $m\ge 2$ and $m>p$,   by H\"{o}lder inequalities
$$
\int_M|\psi|^{q-1}|\xi|\le\left(\int_M|\psi|^{(q-1)mp/(mp-m+p)}dx\right)^{\frac{mp-m+p}{mp}}
\left(\int_M|\xi|^{mp/(m-p)}dx\right)^{\frac{m-p}{mp}}.
$$
Since $(q-1)mp/(mp-m+p)\le \frac{mp}{m-p}$,
we derive
\begin{eqnarray}\label{eq:A.4}
\left.\begin{array}{ll}
&\left(\int_M|\psi|^{(q-1)mp/(mp-m+p)}dx\right)^{\frac{mp-m+p}{(q-1)mp}}\le C(m,q)\|\psi\|_{1,p},\\
&\|\xi\|_{mp/(m-p)}\le C(m)\|\xi\|_{1,p}
\end{array}\right\}
\end{eqnarray}
by Sobolev embedding theorems. These imply that
$\left(1+ 2^{q-1}|\psi|^{q-1}+ 2^{q-1}|\xi|^{q-1} \right)|\xi|$
is integrable. Using the Lebesgue Dominated Convergence Theorem we deduce
\begin{eqnarray*}
\lim_{t\to 0}\frac{\mathcal{H}(x,\psi+t\xi)-\mathcal{H}(x,\psi)}{t}
&=&\int_M\lim_{t\to 0}\frac{H(x,\psi+t\xi)-H(x,\psi)}{t}dx\\
&=&\int_M\langle H_\psi(x,\psi),\xi\rangle  dx.
\end{eqnarray*}
Since  (\ref{eq:A.1}) yields
\begin{eqnarray*}
\int_M|\langle H_\psi(x,\psi),\xi\rangle| dx\leq c_1\int_M\big(1+|\psi|^{q-1}\big)|\xi|dx,
\end{eqnarray*}
(\ref{eq:A.4}) leads to
\begin{eqnarray*}
&&\int_M|\langle H_\psi(x,\psi),\xi\rangle|dx\leq C\big(1+\|\psi\|_{1,p}^{q-1}\big)\|\xi\|_{1,p}.
\end{eqnarray*}
Hence  the G\^ateaux derivative $D\mathcal{H}(\psi)$ exists and
\begin{eqnarray}\label{eq:A.5}
D\mathcal{H}(\psi)\xi=\int_M\langle H_\psi(x,\psi),\xi\rangle dx.
\end{eqnarray}

\noindent{\bf Step 2.}\; {\it $D\mathcal{H}: H^{1,p}(M,\mathbb{S}(M))\to ( H^{1,p}(M,\mathbb{S}(M)))^\ast$ is continuous and thus $\mathcal{H}$
has continuous (Fr\'echet) derivative $\mathcal{H}'=D\mathcal{H}$}.

Let $\psi_n\to \psi$ in $H^{1,p}(M,\mathbb{S}(M))$,
 by the Sobolev embedding theorem yields a positive constant  $c_r>0$, such that
\begin{eqnarray}\label{eq:A.6}
 \|\psi_n-\psi\|_{q} \leq c\|\psi_n-\psi\|_{1,p}\to 0,\quad \text{as}\ n\to\infty.
\end{eqnarray}
By the definition and the H\"older
inequality we have
\begin{eqnarray}\label{eq:A.7}
&&\|D\mathcal{H}(\psi_n)- D\mathcal{H}(\psi)\|_{1,p\ast}=
\sup_{\|\xi\|_{1,p}\le 1}\langle D\mathcal{H}(\psi_n)- D\mathcal{H}(\psi), \xi\rangle\nonumber\\
&=&\sup_{\|\xi\|_{1,p}\le 1}\int_M(H_\psi(x,\psi_n)- H_\psi(x,\psi))\xi dx\nonumber\\
&\le&\left(\sup_{\|\xi\|_{1,p}\le 1}\|\xi\|_{q}\right)
\cdot\left(\int_M|H_\psi(x,\psi_n)- H_\psi(x,\psi)|^{\frac{q}{q-1}}dx\right)^{\frac{q-1}{q}}\nonumber\\
&\le& c\left(\int_M|H_\psi(x,\psi_n)- H_\psi(x,\psi)|^{\frac{q}{q-1}}dx\right)^{\frac{q-1}{q}}.
\end{eqnarray}
Obverse that (\ref{eq:A.1}) implies the
Nemytski operator
$$
H_\psi(x,\cdot): L^{q}(M, \mathbb{S}(M))\to L^{\frac{q}{q-1}}(M, \mathbb{S}(M)),\;
\psi\mapsto H_\psi(x,\psi)
$$
is continuous. This and (\ref{eq:A.7})
yield continuality of $D\mathcal{H}$.



\end{document}